 \documentclass[11pt]{amsart}

\usepackage{amssymb}
\usepackage{bbm}
\usepackage{anysize}
\usepackage{graphicx}
\marginsize{1in}{1in}{1in}{1in}

\usepackage{pdfsync}
\usepackage{color}
\usepackage{comment}
\newtheorem{tm}{Theorem}[section]

\newtheorem{exmp}[tm]{Example}
\newtheorem{lem}[tm]{Lemma}
\newtheorem{assumption}[tm]{Assumption}
\newtheorem{rk}[tm]{Remark}

\numberwithin{equation}{section}

\begin{document}

\title{Time-changed Stochastic Control Problem and its Maximum Principle Theory}


\author{ERKAN NANE}
\address{ERKAN NANE: Department of Mathematics and Statistics,
Auburn University,
Auburn, AL 36849 USA}
\email{ezn0001@auburn.edu}

\author{YINAN NI}
\address{YINAN NI: Department of Mathematics and Statistics,
Auburn University,
Auburn, AL 36849 USA}
\email{yzn0005@auburn.edu}

\begin{abstract}
This paper studies a time-changed stochastic control problem, where the underlying stochastic process is a L\'evy noise time-changed by an inverse subordinator. We establish a maximum principle theory for the time-changed stochastic control problem. We also prove the existence and uniqueness of the corresponding time-changed backward stochastic differential equation involved in the stochastic control problem. Some examples are provided for illustration.
\end{abstract}

\maketitle

\section{Introduction}
Uncertainty is inherent in the real world and changes over time, putting people's decisions at risk. A decision maker wants to select the best choice among all possible ones. The stochastic control theory serves as a tool to such dynamic optimization problem. The world has witnessed many applications of stochastic control theory in various fields such as biology \cite{bio}, economics \cite{econ}, and finance \cite{fin}.

A well known approach to stochastic control problem is based on the maximum principle method. Such method for It\^o diffusion case is first studied by Kushner \cite{kush}, Bismut \cite{bism} and further developed by Bensoussan \cite{bens}, Peng \cite{peng}, and others. The jump diffusion case is formulated by Framstad, {\O}ksendal and Sulem \cite{fos}. The idea of the maximum principle approach is to formulate a Hamiltonian function and derive the adjoint equations, which involve the backward stochastic differential equation. Under sufficient conditions, the optimal control is the solution of a coupled system of forward and backward stochastic differential equations.

The time-changed stochastic differential equation and its related fractional Fokker-Plank equation have become an indispensable tool in applied scientific areas. An example of time-changed stochastic differential equation is $dX(t)=dB(E_t)$ where $X(0)=0$ and $\{E_t,t\geq 0\}$ is the inverse of an $\alpha-stable$ subordinator, see \cite{meer}. The sub-diffusion $B(E_t)$ is governed by time-fractional diffusion equation $\partial_t^\alpha q(x,t) =\partial^2_x q(x,t)$. Some time-changed stochastic differential equations are used to describe real world phenomena. For example, quantitative financial analysts exploit the Black-Scholes framework in derivative pricing, in which the stock price is modeled by Brownian motion. However, some stocks are not actively traded thus their prices stay constant for some time periods. Such phenomenon can be modeled by time-changed Brownian motion but not by the standard Brownian motion, see Figure \ref{constantgraph}. Fruitful studies in this area are available, see \cite{jawy, marc2, erni3, erni}.

\begin{figure}\label{constantgraph}
\begin{center}
\caption{Log price of the Kalev stock \cite{jawy}}
\includegraphics[scale=.66]{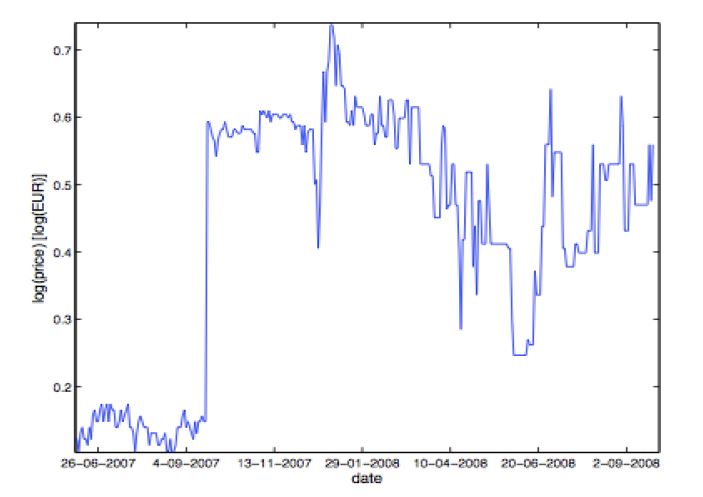}
\end{center}
\end{figure}

As time-changed stochastic processes have been adopted in more and more areas, we believe it is necessary to study the stochastic control problem based on the time-changed stochastic process, which will build up a framework to solve potential optimization problems.  In this paper, we investigate the time-changed stochastic control problem using the maximum principle method. Specifically, we consider the following time-changed stochastic process, see \cite{keib1, erni2}:
\begin{equation}\label{simSDE}
\begin{aligned}
dX(t)&=b(t, E_t, X(t-), u(t))dE_t+\sigma(t, E_t, X(t-), u(t))dB_{E_t}\\
&+\int_{|y|<c}\gamma(t, E_t, X(t-), u(t),y)\tilde{N}(dE_t,dy),
\end{aligned}
\end{equation}
with $X(0)=x_0\neq 0$ and the corresponding performance function
\begin{equation}
J(u)=\mathbb{E} \Big[\int_0^Tg(t, E_t, X(t), u(t))dE_t+h(X(T))\Big], \ u\in \mathcal{A} ,
\end{equation}
where $u(t)=u(t,w)\in U \subset \mathbb{R}$ is the control and $\mathcal{A}$ denotes the set of $admissible$ controls. We establish a maximum principle theory for the stochastic control problem to find $u^*\in \mathcal{A}$ such that
\begin{equation}
J(u^*)=\sup_{u\in \mathcal{A}}J(u).
\end{equation}
Then we extend such result to a more general time-changed stochastic process involving time drift term $dt$:
\begin{equation}\label{SDE}
\begin{aligned}
dX(t)&=\mu(t, E_t, X(t-), u(t))dt+b(t, E_t, X(t-), u(t))dE_t+\sigma(t, E_t, X(t-), u(t))dB_{E_t}\\
&+\int_{|y|<c}\gamma(t, E_t, X(t-), u(t),y)\tilde{N}(dE_t,dy),
\end{aligned}
\end{equation}
with $X(0)=x_0\neq 0$,
and the corresponding performance function
\begin{equation}
J(u)=\mathbb{E}  \Big[\int_0^Tf(t, E_t, X(t), u(t))dt+\int_0^Tg(t, E_t, X(t), u(t))dE_t+h(X(T))\Big], \ u\in \mathcal{A} .
\end{equation}

In the remaining parts of this paper, some necessary concepts and preliminary results will be given in Section 2. In section 3 and 4, we establish  a maximum principle theory for time-changed stochastic control problems mentioned above and provide some examples for illustration.

\section{Preliminaries}

Let $(\Omega, \mathcal{F},(\mathcal{F}_t),P)$ be a filtered probability space satisfying usual hypotheses of completeness and right continuity. Assume that an independent $\mathcal{F}_t$-adapted Poisson random measure $N$ is defined on $\mathbb{R}_+\times (\mathbb{R}-\{0\})$ with compensator $\tilde{N}$ and intensity measure $\nu$, where $\nu$ is a L\'evy measure such that $\tilde{N}(dt,dy)=N(dt,dy)-\nu(dy)dt$ and $\int_{\mathbb{R}-\{0\}}(|y|^2 \land 1)\nu(dy)<\infty$.

Let $\{D(t),t\geq 0\}$ be a right continuous with left limits (RCLL) subordinator starting from 0 with Laplace transform
\begin{equation}
\mathbbm{E}e^{-\lambda D(t)}=e^{-t\phi(\lambda)},
\end{equation}
where Laplace exponent $\phi(\lambda)=\int_0^\infty(1-e^{-\lambda x})\nu(dx)$
, define its inverse
\begin{equation}\label{invstable}
E_t:=\inf\{ \tau>0: D(\tau)>t\}.
\end{equation}

\begin{lem} (Lemma 8 in \cite{keib1})
Let E be the inverse of a subordinator D with Laplace exponent $\phi$ and infinite L\'evy measure. Then $\mathbb{E}[e^{\lambda E_t}]<\infty$, $\forall \lambda \in \mathbb{R}$ and $t\geq 0$. In particular, for each $t>0$, moments of $E_t$ of all orders exist and are given by
\begin{equation}
\mathbb{E}[E^n_t]=\mathcal{L}_s^{-1}\Big[\frac{n!}{s\phi^n(s)}\Big](t),  n\in \mathbb{N},
\end{equation}
where $\mathcal{L}_s^{-1}[g(s)]$ denotes the inverse Laplace transform of a function $g(s)$.
\end{lem}

Consider the following time-changed stochastic differential equation:
\begin{equation}\label{simSDE}
\begin{aligned}
dX(t)=&b(t, E_t, X(t-), u(t))dE_t+\sigma(t, E_t, X(t-), u(t))dB_{E_t}\\
&+\int_{|y|<c}\gamma(t, E_t, X(t-), u(t),y)\tilde{N}(dE_t,dy),
\end{aligned}
\end{equation}
with $X(0)=x_0\neq 0$,
where $b,\sigma,\gamma$ are real-valued functions satisfying the following Lipschitz condition \ref{lip} and assumption \ref{tec}
such that there exists a unique $\mathcal{G}_t$-adapted process $X(t)$ satisfying time-changed SDE \eqref{simSDE}, see Lemma 4.1 in \cite{keib}. The filtration $\{\mathcal{G}_t\}_{t\geq 0}$ is defined as
\begin{equation}
\mathcal{G}_t = \bigcap_{u>t}\{[\mathcal{F}_y: 0\leq y\leq u ]\vee \sigma[E_y: y\geq 0] \}.
\end{equation}

\begin{assumption}\label{lip}
(Lipschitz condition)
There exists a positive constant K such that
\begin{equation}
\begin{aligned}
&\Big|b(t_1,t_2,x,u)-b(t_1,t_2,y,u)\Big|^2+\Big|\sigma(t_1,t_2,x,u)-\sigma(t_1,t_2,y,u)\Big|^2\\
&+\int_{|z|<c}\Big|\gamma(t_1,t_2,x,u,z)-\gamma(t_1,t_2,y,u,z)\Big|^2\nu(dz)\leq K|x-y|^2,
\end{aligned}
\end{equation}
for all $t_1,t_2\in \mathbb{R}_+$ and $x,y\in \mathbb{R}$.
\end{assumption}

\begin{assumption}\label{tec}
If $X(t)$ is a RCLL and $\mathcal{G}_t$-adapted process, then
\begin{equation}
b(t, E_t, X(t), u(t)), \sigma(t, E_t, X(t), u(t)),\gamma(t, E_t, X(t), u(t),y)\in \mathcal{L}(\mathcal{G}_t),
\end{equation}
where  $\mathcal{L}(\mathcal{G}_t)$ denotes the class of left continuous with right limits (LCRL) and $\mathcal{G}_t$-adapted processes.
\end{assumption}

The process $u(t)=u(t,w)\in U \subset \mathbb{R}$ is the control. Assume that $u$ is adapted and RCLL, and that the corresponding equation \eqref{simSDE} has a unique strong solution $X^{(u)}(t), t\in [0,T]$. Such controls are called $admissible$. The set of admissible controls is denoted by $\mathcal{A}$.

\begin{lem}(It\^o Formula for Time-Changed L\'evy Noise, Lemma 3.1 in \cite{erni2})\label{itofor}
Let $D(t)$ be a RCLL subordinator and $E_t$ its inverse process as \eqref{invstable}.  Let $X$ be a process defined as following:
\begin{equation}
\begin{aligned}\label{sdelevy}
X(t)=&x_0+\int_0^t\mu(t, E_t, X(t-))dt+\int_0^tb(t, E_t, X(t-))dE_t+\int_0^t\sigma(t, E_t, X(t-))dB_{E_t}\\
&+\int_0^t\int_{|y|<c}\gamma(t, E_t, X(t-),y)\tilde{N}(dE_t,dy),
\end{aligned}
\end{equation}
where $\mu, b, \sigma, \gamma$ are measurable functions such that all integrals are defined. Here $c$ is the maximum allowable jump size.\\
Then, for all $F : \mathbb{R}_+\times \mathbb{R}_+\times \mathbb{R}\rightarrow \mathbb{R}$ in $C^{1,1,2}(\mathbb{R}_+\times \mathbb{R}_+\times \mathbb{R},\mathbb{R})$, with probability one,
\begin{equation}
\begin{aligned}\label{itolevy}
F(t, E_t, &X(t))-F(0,0,x_0)=\int_0^t L_1F(s, E_s, X(s-))ds+\int_0^t L_2F(s, E_s, X(s-))dE_s\\
&+\int_0^t\int_{|y|<c}\Big[F(s, E_s, X(s-)+\gamma(s, E_s, X(s-),y))-F(s, E_s, X(s-))\Big]\tilde{N}(dE_s,dy)\\
&+\int_0^t F_x(s, E_s, X(s-))\sigma(s, E_s, X(s-))dB_{E_s},
\end{aligned}
\end{equation}
where
\begin{equation}
\begin{aligned}\label{linearop}
L_1F(t_1,&t_2,x)=F_{t_1}(t_1,t_2,x)+F_{x}(t_1,t_2,x)\mu(t_1,t_2,x),\\
L_2F(t_1,&t_2,x)=F_{t_2}(t_1,t_2,x)+F_{x}(t_1,t_2,x)b(t_1,t_2,x)+\frac{1}{2}F_{xx}(t_1,t_2,x)\sigma^2(t_1,t_2,x)\\
+&\int_{|y|<c}\Big[F(t_1,t_2,x+\gamma(t_1,t_2,x,y))-F(t_1,t_2,x)-F_x(t_1,t_2,x)\gamma(t_1,t_2,x,y)\Big]\nu(dy).
\end{aligned}
\end{equation}
\end{lem}

\begin{lem}\label{exunbsde}
(Existence and Uniqueness of BSDE)

Consider the following time-changed Backward stochastic differential equation

\begin{equation}\label{BSDE}
dX(t)=-\mu(t,E_t, X(t-),u(t))dE_t+u(t)dB_{E_t}+\int_{\mathbb{R} \setminus \{0\}}h(t,z)\tilde{N}(dE_t,dz),
\end{equation}
with $X(T)=X$, where $\mu\in L^2(\mathbb{R}_+,\mathbb{R}_+,\mathbb{R},\mathbb{R}), h\in L^2(\mathbb{R}_+,\mathbb{R})$. If there exists a positive constant $L_\mu>0$ such that $|\mu(t_1,t_2,x_1,u_1)-\mu(t_1,t_2,x_2,u_2)|\leq L_\mu\Big( |x_1-x_2|+|u_1-u_2| \Big)$,
then there exists a unique solution $(X(t),u(t))$ of \eqref{BSDE}.
\end{lem}

\begin{proof}
To prove the uniqueness, suppose $(X_1(t),u_1(t))$ and $(X_2(t),u_2(t))$ are two solutions to \eqref{BSDE} in $L^2(\Omega\times\mathbb{R}_+)\times L^2(\Omega\times\mathbb{R}_+)$. By It\^o formula,
\begin{equation}\label{uniu}
\begin{aligned}
\Big|&X_1(T)-X_2(T) \Big|^2-\Big|X_1(t)-X_2(t) \Big|^2=\int_t^T|u_1(s)-u_2(s)|^2dE_s\\
&+\int_t^T2(X_1(s)-X_2(s))\Big[-\Big(\mu(s,E_s,X_1(s),u_1(s))-\mu(s,E_s,X_2(s),u_2(s))\Big)dE_s+\big(u_1(s)-u_2(s)\big)dB_{E_s}\Big]\\
\end{aligned}
\end{equation}

Thus,
\begin{equation}
\begin{aligned}
\Big|&X_1(t)-X_2(t) \Big|^2+\int_t^T|u_1(s)-u_2(s)|^2dE_s+\int_t^T2(X_1(s)-X_2(s)) \big(u_1(s)-u_2(s)\big)dB_{E_s}\\
=&\int_t^T2(X_1(s)-X_2(s))\Big(\mu(s,E_s,X_1(s),u_1(s))-\mu(s,E_s,X_2(s),u_2(s))\Big)dE_s\\
\leq & \int_t^T 2L_\mu |X_1(s)-X_2(s)|\Big( |X_1(s)-X_2(s)|+|u_1-u_2| \Big)dE_s \\
\leq & \int_t^T 2L_\mu \Big[ |X_1(s)-X_2(s)|^2+ \frac{L_\mu}{2}|X_1(s)-X_2(s)|^2+\frac{1}{2L_\mu}|u_1(s)-u_2(s)|^2 \Big]dE_s\\
=&(2L_\mu+L_\mu^2) \int_t^T|X_1(s)-X_2(s)|^2dE_s+\int_t^T|u_1(s)-u_2(s)|^2 dE_s.
\end{aligned}
\end{equation}

Take expectations on both sides,
\begin{equation} \label{inequ}
\mathbb{E}\Big[\Big|X_1(t)-X_2(t) \Big|^2\Big]\leq (2L_\mu+L_\mu^2)\mathbb{E} \Big[\int_t^T|X_1(s)-X_2(s)|^2dE_s\Big].
\end{equation}

Note that we apply Martingale property to derive inequality \eqref{inequ} and lay some details below.
\begin{equation}
\begin{aligned}
\int_t^T(X_1(s)-X_2(s)) &\big(u_1(s)-u_2(s)\big)dB_{E_s}=\int_0^\infty 1{\{ t \leq s\leq T \}}(X_1(s)-X_2(s)) \big(u_1(s)-u_2(s)\big)dB_{E_s}\\
=&\int_0^\infty 1_{\{t \leq D(s-)\leq T \}}(X_1(D(s-))-X_2(D(s-))) \big(u_1(D(s-))-u_2(D(s-))\big)dB_s,
\end{aligned}
\end{equation}
since $(X_1(t),u_1(t))$ and $(X_2(t),u_2(t))$ are in $L^2(\Omega\times\mathbb{R_+})$,
\begin{equation}
\begin{aligned}
&\mathbb{E}\int_0^\infty\Big| 1_{\{t \leq D(s-)\leq T\}}(X_1(D(s-))-X_2(D(s-))) \big(u_1(D(s-))-u_2(D(s-))\big)\Big|^2ds\\
& \leq  \mathbb{E}\int_0^\infty\Big|(X_1(D(s-))-X_2(D(s-))) \big(u_1(D(s-))-u_2(D(s-))\big)\Big|^2ds<\infty,
\end{aligned}
\end{equation}
we have
\begin{equation}
\begin{aligned}
&\mathbb{E}\int_t^T(X_1(s)-X_2(s)) \big(u_1(s)-u_2(s)\big)dB_{E_s}\\
&=\mathbb{E}\int_0^\infty 1_{\{t \leq D(s-)\leq T \}}(X_1(D(s-))-X_2(D(s-))) \big(u_1(D(s-))-u_2(D(s-))\big)dB_s\\
&=0.
\end{aligned}
\end{equation}

Next we apply time-changed Gronwall's method by Lemma 3.1 in \cite{wu}. Define $F(t)=\int_t^T|X_1(s)-X_2(s)|^2dE_s$, then $F(T)=0$ and
\begin{equation}
\begin{aligned}
-d\Big(F(t)\exp(kE_t)\Big)&=-\exp(kE_t)dF(t)-k\exp(kE_t)F(t)dE_t\\
&=\exp(kE_t)\Big(\Big|X_1(t)-X_2(t)\Big|^2-k\int_t^T\Big|X_1(s)-X_2(s)\Big|^2dE_s\Big)dE_t,
\end{aligned}
\end{equation}
thus
\begin{equation}
-F(T)\exp(kE_T)+F(t)\exp(kE_t)=\int_t^T\Big[\exp(kE_s)\Big(\Big|X_1(s)-X_2(s)\Big|^2-k\int_s^T\Big|X_1(u)-X_2(u)\Big|^2dE_u\Big)\Big]dE_s.
\end{equation}

Taking expectations and letting $k=2L_\mu+L_\mu^2$ imply that
\begin{equation}
\begin{aligned}
&\mathbb{E}\Big[F(t)\exp(kE_t)\Big]=\mathbb{E}\Big[\int_t^T\exp(kE_s)\Big(\Big|X_1(s)-X_2(s)\Big|^2-k\int_s^T\Big|X_1(u)-X_2(u)\Big|^2dE_u\Big)dE_s\Big]\\
&=\mathbb{E}\Big[\mathbb{E}\Big[\int_t^T\exp(kE_s)\Big(\Big|X_1(s)-X_2(s)\Big|^2-k\int_s^T\Big|X_1(u)-X_2(u)\Big|^2dE_u\Big)dE_s\Big]\Big|\sigma\{E_s,s\in(t,T)\}\Big]\\
&=\mathbb{E}\Big[\int_t^T\exp(kE_s)\mathbb{E}\Big(\Big|X_1(s)-X_2(s)\Big|^2-k\int_s^T\Big|X_1(u)-X_2(u)\Big|^2dE_u\Big)dE_s\Big|\sigma\{E_s,s\in(t,T)\}\Big]\\
&\leq 0
\end{aligned}
\end{equation}

It follows that
\begin{equation}
\mathbb{E} \Big[F(t)\Big]\leq\mathbb{E}\Big[F(t)\exp(kE_t)\Big]\leq 0,
\end{equation}
so $X_1(s)=X_2(s)$ a.s. for $\forall s\in (t,T)$. By \eqref{uniu}, since $X_1(s)=X_2(s)$ a.s. for $\forall s\in (t,T)$, we have $\int_t^T|u_1(s)-u_2(s)|^2dE_s = 0$, thus   $u_1(s)=u_2(s)$ a.s. for $\forall s\in (t,T)$.  The uniqueness is proved.\\

To prove the existence, let $u_0(t)=0$, $\{(X_n(t),u_n(t)); 0\leq t\leq T\}_{n\geq 1}$ be a sequence defined recursively by
\begin{equation}
X_{n-1}(t)+\int_t^T\mu(s,E_s,X_{n-1}(s),u_{n-1}(s))dE_s-\int_t^T u_{n-1}(s)dB_{E_s}-\int_t^T\int_{\mathbb{R} \setminus \{0\}}h(s,z)\tilde{N}(dE_s,dz)=X_n.
\end{equation}
Then
\begin{equation}
\begin{cases}
   & dX_n(t)=-\mu(t,E_t,X_{n-1}(t),u_{n-1}(t))dE_t+u_{n-1}(t)dB_{E_t}+ \int_{\mathbb{R} \setminus \{0\}}h(t,z)\tilde{N}(dE_t,dz),\\

   & d X_{n+1}(t)=-\mu(t,E_t,X_{n}(t),u_{n}(t))dE_t+u_{n}(t)dB_{E_t}+ \int_{\mathbb{R} \setminus \{0\}}h(t,z)\tilde{N}(dE_t,dz),\\

   &X_n(T)=X_{n+1}(T)=X.

   \end{cases}
\end{equation}
By It\^o formula in Lemma \ref{itofor}, there exists $k>0$ such that
\begin{equation}
\begin{aligned}
&\Big| X_{n+1}(t)-X_n(t)\Big|^2+\int_t^T(u_{n}(s)-u_{n-1}(s))^2dE_s+2\int_t^T(X_{n+1}(s)-X_n(s))(u_{n}(s)-u_{n-1}(s))dB_{E_s}\\
&=2\int_t^T(X_{n+1}(s)-X_n(s))\Big(\mu(s,E_s,X_{n}(s),u_{n}(s))-\mu(s,E_s,X_{n-1}(s),u_{n-1}(s))  \Big)dE_s\\
&\leq 2L_\mu \int_t^T|X_{n+1}(s)-X_n(s)|\Big(|X_n(s)-X_{n-1}(s)|+|u_{n}(s)-u_{n-1}(s)|\Big)dE_s\\
&\leq k\Big[ \int_t^T|X_{n+1}(s)-X_n(s)|^2dE_s+ \int_t^T|X_n(s)-X_{n-1}(s)|^2dE_s\Big]+\frac{1}{2}\int_t^T|u_{n}(s)-u_{n-1}(s)|^2dE_s.
\end{aligned}
\end{equation}
Taking expectation on both sides implies
\begin{equation}\label{exist1}
\begin{aligned}
\mathbb{E}\Big| X_{n+1}(t)-X_n(t)\Big|^2&+\frac{1}{2}\mathbb{E}\int_t^T|u_{n}(s)-u_{n-1}(s)|^2dE_s\\
\leq & k\mathbb{E}\Big[ \int_t^T|X_{n+1}(s)-X_n(s)|^2dE_s+ \int_t^T|X_n(s)-X_{n-1}(s)|^2dE_s\Big].
\end{aligned}
\end{equation}

Define $F_n(t)=\int_t^T\Big| X_{n}(s)-X_{n-1}(s)\Big|^2dE_s$ for all $n\geq 1$, then $F_n(T)=0$ and
\begin{equation}
\begin{aligned}
-d\Big(F_{n+1}(t)\exp(kE_t)\Big)&=-\exp(kE_t)dF_{n+1}(t)-k\exp(kE_t)F_{n+1}(t)dE_t\\
&=\exp(kE_t)\Big[\Big|X_{n+1}(t)-X_n(t)\Big|^2-k\int_t^T\Big|X_{n+1}(s)-X_n(s)\Big|^2dE_s\Big]dE_t,
\end{aligned}
\end{equation}
By a similar argument for uniqueness and using \eqref{exist1},
\begin{equation}
\begin{aligned}
&\mathbb{E}\Big[F_{n+1}(t)\exp(kE_t)\Big]=\mathbb{E}\Big[\int_t^T\exp(kE_s)\Big[\Big|X_{n+1}(s)-X_n(s)\Big|^2-k\int_s^T\Big|X_{n+1}(l)-X_n(l)\Big|^2dE_l\Big]dE_s\Big]\\
&=\mathbb{E}\Big[\mathbb{E}\Big[\int_t^T\exp(kE_s)\Big[\Big|X_{n+1}(s)-X_n(s)\Big|^2-k\int_s^T\Big|X_{n+1}(l)-X_n(l)\Big|^2dE_l\Big]dE_s\Big]\Big|\{\sigma(E_s,s\in(t,T))\}\Big]\\
&=\mathbb{E}\Big[\int_t^T\exp(kE_s)\mathbb{E}\Big[\Big|X_{n+1}(s)-X_n(s)\Big|^2-k\int_s^T\Big|X_{n+1}(l)-X_n(l)\Big|^2dE_l\Big]dE_s\Big|\{\sigma(E_s,s\in(t,T))\}\Big]\\
&\leq \mathbb{E}\Big[\int_t^T\exp(kE_s)k\mathbb{E}\Big[\int_s^T|X_n(l)-X_{n-1}(l)|^2dE_l\Big]  dE_s\Big|\{\sigma(E_s,s\in(t,T))\}\Big]\\
&= \mathbb{E}\Big[\int_t^Tk\exp(kE_s)\mathbb{E}\Big[F_n(s)\Big]  dE_s\Big|\{\sigma(E_s,s\in(t,T))\}\Big]\\
&=\mathbb{E}\Big[\int_t^Tk\exp(kE_s)F_n(s) dE_s\Big],
\end{aligned}
\end{equation}
letting $t=0$,
\begin{equation}
\mathbb{E}F_{n+1}(0)\leq \mathbb{E}\int_0^Tke^{kE_s}F_n(s)dE_s\leq \mathbb{E}\Big[\Big(e^{kE_T}\Big)^n\frac{F_1(0)}{n!}\Big]\rightarrow 0,\  as\  n\rightarrow \infty.
\end{equation}
Thus, $\{X_n\}$ is a Cauchy sequence in $L^2(\Omega\times\mathbb{R}_+)$. Taking \eqref{exist1} into consideration,  $\{u_n\}$ is also a Cauchy sequence in $L^2(\Omega\times\mathbb{R}_+)$. Thus, the existence of solution to \eqref{BSDE} is proved.
\end{proof}

\section{Time-changed Stochastic Control Problem}

In this section, we solve the time-changed stochastic control problem through the maximum principle approach. An example is provided to illustrate how our method works for a particular time-changed stochastic problem.\\

We consider a performance criterion $J=J(u)$ of the form
\begin{equation}\label{performance1}
J(u)=\mathbb{E}  \Big[\int_0^Tg(t, E_t, X(t), u(t))dE_t+h(X(T))\Big], \ u\in \mathcal{A} ,
\end{equation}
where $g:[0,T]\times \mathbb{R}_+ \times \mathbb{R} \times U   \rightarrow \mathbb{R}$ is continuous, $h:\mathbb{R}\rightarrow \mathbb{R}$ is $C^1, T<\infty $ is a fixed deterministic time and
\begin{equation}
\mathbb{E} \Big[\int_0^Tg(t, E_t, X(t), u(t))dE_t+h(X(T))\Big]<\infty, \forall u\in \mathcal{A}.
\end{equation}

The stochastic control problem is to find the optimal control $u^*\in \mathcal{A}$ such that
\begin{equation}\label{supj}
J(u^*)=\sup_{u\in \mathcal{A}}J(u).
\end{equation}

Since $E_t$ is right continuous and nondecreasing, $\frac{dE_t}{dt}$ exists for $t\geq 0$ a.e. \\

Define the $Hamiltonian\ H:[0,T]\times\mathbb{R_+}\times\mathbb{R}\times U\times\mathbb{R}\times \mathbb{R}\times\mathcal{R}$ by
\begin{equation}\label{hamil}
\begin{aligned}
H(t_1,t_2, x, u,p,q,r)=&g(t_1,t_2, x, u)+pb(t_1,t_2, x, u)+q\sigma(t_1,t_2, x, u)\\
&+\int_{\mathbb{R}}\gamma(t_1,t_2, x, u,z)r(t_2,z)\nu(dz),
\end{aligned}
\end{equation}
or
\begin{equation}
\begin{aligned}
H(t,E_t, X(t), u(t),p(t),&q(t),r(t,z))=g(t,E_t, X(t), u(t))+p(t)b(t,E_t,X(t),u(t))\\
&+q(t)\sigma(t,E_t, X(t), u(t))+\int_{\mathbb{R}}\gamma(t,E_t, X(t), u(t),z)r(E_t,z)\nu(dz),
\end{aligned}
\end{equation}
 where $\mathcal{R}$ is the set of functions $r:\mathbb{R_+}\times\mathbb{R} \to \mathbb{R}$ such that the integrals in \eqref{hamil} exists.

Define the adjoint equation in the unknown processes $p(t) \in \mathbb{R},\ q(t)\in \mathbb{R}$, and $r(t,z)\in \mathbb{R}$ in the backward stochastic differential equations
\begin{equation}\label{adj1}
\begin{aligned}
dp(t)=&-H_x(t,E_t,X(t),u(t),p(t),q(t),r(t,\cdot))dE_t\\
&+q(t)dB_{E_t}+\int_{\mathbb{R}}r(E_t,z)\tilde{N}(dE_t,dz), t<T\\
p(T)=&h_x(X(T)).
\end{aligned}
\end{equation}

\begin{tm}\label{tcmpt}
(Time-Changed Maximum Principle Theorem) Let $\hat{u}\in \mathcal{A}$ with corresponding solution $\hat{X}=X^{(\hat{u})}$ of \eqref{simSDE} and suppose there exists a solution $(\hat{p}(t), \hat{q}(t), \hat{r}(t,z))$ of the corresponding adjoint equation \eqref{adj1} satisfying
\begin{equation}
\mathbb{E}\Big[ \int_0^T (\hat{X}(t)-X^{(u)}(t))^2\Big(\hat{q}^2(t)+\int_{\mathbb{R}}\hat{r}^2(E_t,z)\nu(dz) \Big)dE_t\Big]<\infty
\end{equation}
and
\begin{equation}
\mathbb{E}\Big[ \int_0^T \hat{p}^2(t)\Big(\sigma^2(t,E_t,X^{(u)}(t),u(t) )+\int_{\mathbb{R}}\gamma^2(t,E_t,X^{(u)}(t),u(t),z)\nu(dz) \Big)dE_t\Big]<\infty, \forall u\in\mathcal{A}.
\end{equation}
Moreover, suppose that
\begin{equation}
H(t,E_t,\hat{X}(t),\hat{u}(t),\hat{p}(t), \hat{q}(t), \hat{r}(t,\cdot))=\sup_{v\in U} H(t,E_t,\hat{X}(t),v,\hat{p}(t), \hat{q}(t), \hat{r}(t,\cdot))
\end{equation}
for all $t$, that $h(x)$ in \eqref{performance1} is a concave function of $x$ and that
\begin{equation}
\hat{H}(x):=\max_{v\in U} H(t_1,t_2,x,v,\hat{p}(t), \hat{q}(t), \hat{r}(t,\cdot))
\end{equation}
exists and is a concave function of $x$ for all $t\in[0,T]$.
Then $\hat{u}$ is an optimal control of stochastic control problem \eqref{supj}.
\end{tm}

\begin{proof}
Let $u \in \mathcal{A}$ be an admissible control with corresponding state process $X(t)=X^{(u)}(t)$. We would like to show that
\begin{equation}
J(\hat{u})-J(u)=\mathbb{E}\Big[ \int_0^T{g(t,E_t,\hat{X}(t),\hat{u}(t))-g(t,E_t,X(t),u(t))}dt+h(\hat{X}(T))-h(X(T)) \Big]\geq 0.
\end{equation}

Since $g$ is concave, using It\^o formula \eqref{itolevy},
\begin{equation}
\begin{aligned}
&\mathbb{E}[h(\hat{X}(T))-h(X(T))]\geq  \mathbb{E}[h_x(\hat{X}(T))(\hat{X}(T)-X(T))]=\mathbb{E}[(\hat{X}(T)-X(T))\hat{P}(T)]\\
&=\mathbb{E}\Big[ \int_0^T (\hat{X}(t)-X(t))d\hat{p}(t)+ \int_0^T \hat{p}(t) d(\hat{X}(t)-X(t))+ \int_0^T d\hat{p}(t) d(\hat{X}(t)-X(t))\Big]\\
&=\mathbb{E}\Big[ \int_0^T (\hat{X}(t)-X(t))d\hat{p}(t)+ \int_0^T \hat{p}(t) d(\hat{X}(t)-X(t))\\
&\ \ \ + \int_0^T \hat{q}(t)\Big(\sigma(t,E_t,\hat{X}(t),\hat{u}(t))-\sigma(t,E_t,X(t),u(t))\Big)dE_t\\
&\ \ \ +\int_0^T\int_{\mathbb{R}}\hat{r}(t,z)\Big(\gamma(t,E_t,\hat{X}(t),\hat{u}(t))-\gamma(t,E_t,X(t),u(t))\Big)\nu(dz)dE_t\Big].
\end{aligned}
\end{equation}

Among above terms,
\begin{equation}
\mathbb{E}\Big[\int_0^T\hat{p}(t)d(\hat{X}(t)-X(t))\Big]=\mathbb{E}\Big[\int_0^T \hat{p}(t)\Big(b(t,E_t,\hat{X}(t),u(t))-b(t,E_t,X(t),u(t))\Big)dE_t\Big]
\end{equation}

Thus,
\begin{equation}\label{equation1}
\begin{aligned}
J(\hat{u})-J(u)=&\mathbb{E}\Big[ \int_0^T (\hat{X}(t)-X(t))d\hat{p}(t)+ \int_0^T g(t,E_t,\hat{X}(t),\hat{u}(t))-g(t,E_t,X(t),u(t))dE_t\\
&+\int_0^T \hat{p}(t)\Big(b(t,E_t,\hat{X}(t),\hat{u}(t))-b(t,E_t,X(t),u(t))\Big)dE_t\\
&+ \int_0^T \hat{q}(t)\Big(\sigma(t,E_t,\hat{X}(t),\hat{u}(t))-\sigma(t,E_t,X(t),u(t))\Big)dE_t\\
&+\int_0^T\int_{\mathbb{R}}\hat{r}(t,z)\Big(\gamma(t,E_t,\hat{X}(t),\hat{u}(t))-\gamma(t,E_t,X(t),u(t))\Big)\nu(dz)dE_t\Big].
\end{aligned}
\end{equation}

In addition,
\begin{equation}\label{equation2}
\begin{aligned}
&H(t,E_t,\hat{X}(t),\hat{u}(t),\hat{p}(t),\hat{q}(t),\hat{r}(t))-H(t,E_t,X(t),u(t),\hat{p}(t),\hat{q}(t),\hat{r}(t))\\
=&\Big(g(t,E_t,\hat{X}(t),\hat{u}(t))-g(t,E_t,X(t),u(t))\Big)+\hat{p}(t)\Big(b(t,E_t,\hat{X}(t),\hat{u}(t))-b(t,E_t,X(t),u(t))\Big)\\
&+\hat{q}(t)\Big(\sigma(t,E_t,\hat{X}(t),\hat{u}(t))-\sigma(t,E_t,X(t),u(t))\Big)\\
&+\int_{\mathbb{R}}\hat{r}(t,z)\Big( \gamma(t,E_t,\hat{X}(t),\hat{u}(t))- \gamma(t,E_t,X(t),u(t))\Big)\nu(dz),
\end{aligned}
\end{equation}

and by \eqref{adj1} we have
\begin{equation}\label{equation3}
\begin{aligned}
&(\hat{X}(t)-X(t))d\hat{p}(t)=\hat{X}(t)d\hat{p}(t)-X(t)d\hat{p}(t)\\
=&\hat{X}(t)\Big[-H_x(t,E_t,\hat{X}(t),\hat{u}(t),\hat{p}(t),\hat{q}(t),\hat{r}(t,\cdot))dE_t+\hat{q}(t)dB_{E_t}+\int_{\mathbb{R}}\hat{r}(t,z)\tilde{N}(dE_t,dz)\Big]\\
&-X(t)\Big[-H_x(t,E_t,\hat{X}(t),\hat{u}(t),\hat{p}(t),\hat{q}(t),\hat{r}(t,\cdot))dE_t+\hat{q}(t)dB_{E_t}+\int_{\mathbb{R}}\hat{r}(t,z)\tilde{N}(dE_t,dz)\Big]\\
=&-(\hat{X}(t)-X(t))H_x(t,E_t,\hat{X}(t),\hat{u}(t),\hat{p}(t),\hat{q}(t),\hat{r}(t,\cdot))dE_t\\
&+(\hat{X}(t)-X(t))(\hat{q}(t)dB_{E_t}+ \int_{\mathbb{R}}\hat{r}(t,z)\tilde{N}(dE_t,dz)).\\
\end{aligned}
\end{equation}

Then, since $H$ is concave in $x$, putting equations \eqref{equation2} and \eqref{equation3} into \eqref{equation1} and following the proof in \cite{fos}, we get
\begin{equation}
\begin{aligned}
J(\hat{u})-J(u)=&\mathbb{E}\Big[ \int_0^T -(\hat{X}(t)-X(t))H_x(t,E_t,\hat{X}(t),\hat{u}(t),\hat{p}(t),\hat{q}(t),\hat{r}(t,\cdot))dE_t\\
&+\int_0^TH(t,E_t,\hat{X}(t),\hat{u}(t),\hat{p}(t),\hat{q}(t),\hat{r}(t,\cdot))-H(t,E_t,X(t),u(t),\hat{p}(t),\hat{q}(t),\hat{r}(t,\cdot))dE_t\\
&\geq 0.
\end{aligned}
\end{equation}

\end{proof}

\begin{exmp}(The Time-Changed Stochastic Linear Regulator Problem)\\
  The Linear Regulator Problem aims to reduce the amount of work or energy consumed by the control system to optimize the controller. In this example, we consider the following time-changed stochastic linear regulator problem:
\begin{equation}
\Phi(x_0)=\inf_{u\in \mathcal{A}}\mathbb{E}\Big[\int_0^T\frac{X^2(t)+u^2(t)}{2}dE_t+\lambda X^2(T)\Big],
\end{equation}
where
\begin{equation}
dX(t)=u(t)dE_t+\sigma dB_{E_t}+\int_\mathbb{R}z\tilde{N}(dE_t,dz),\ X(0)=x_0.
\end{equation}

Construct the $Hamiltonian$:
\begin{equation}
H(t_1,t_2,x,u,p,q,r)=\frac{x^2+u^2}{2}+pu+\sigma q+\int_\mathbb{R}\gamma z \nu(dz).
\end{equation}
The adjoint equations are
\begin{equation}\label{ex1p1}
\begin{cases}
   & dp(t)=-X(t)dE_t+q(t)dB_{E_t}+\int_\mathbb{R}r (E_t, z) \tilde{N}(dE_t,dz),\\

   & p(T)=2\lambda X(T).
   \end{cases}
\end{equation}
The first and second order condition implies that $Hamiltonian:\ H(t_1,t_2,x,u,p,q,r)$ achieves the minimum at $u^*(t)=-p(t)$.

To find an explicit solution of $u^*(t)$, suppose $p(t)=h(E_t)X(t)$, where $h:\mathbb{R}_+\rightarrow \mathbb{R}_+$. Then $u^*(t)=-h(E_t)X(t)$ and
\begin{equation}\label{ex1p2}
\begin{aligned}
dp(t)&=h(E_t)dX(t)+h'(E_t)X(t)dE_t\\
&=h(E_t)\Big(u(t)dE_t+\sigma dB_{E_t}+\int_\mathbb{R}z\tilde{N}(dE_t,dz)\Big)+h'(E_t)X(t)dE_t\\
&=X(t)(-h^2(E_t)+h'(E_t))dE_t+h(E_t)\sigma dB_{E_t}+h(E_t)\int_\mathbb{R}z\tilde{N}(dE_t,dz).
\end{aligned}
\end{equation}

Compare \eqref{ex1p1} and \eqref{ex1p2}, $-h^2(E_t)+h'(E_t)=-1$ and $h(E_T)=2\lambda$. The general solution to this ordinary differential equation gives
\begin{equation}
h(E_t)=-\frac{2\lambda-1+(2\lambda+1)e^{2(E_t-E_T)}}{2\lambda-1-(2\lambda+1)e^{2(E_t-E_T)}}.
\end{equation}
Thus, we have the explicit formula for the optimal control $u^*(t)=-h(E_t)X(t)$. Similarly, $q(t)=h(E_t)\sigma$ and $r(E_t,z)=h(E_t)z$. A simulation of the optimal control $u^*(t)$ with $\lambda=-\frac{1}{2}, \sigma = 1, x_0=-.01$, standard normal distribution $\nu$, and inverse stable subordinator $E(t)$ having $\alpha = .9$ is displayed in Figure \ref{fig:EXAMPLE1_9}.

Keeping  all others parts the same as in the figure \ref{fig:EXAMPLE1_9}, we also simulate the optimal control $u^*(t)$ for $\alpha = .7$ and $\alpha = .5$ in Figure \ref{fig:EXAMPLE1_7} and \ref{fig:EXAMPLE1_5}, respectively.
Overall, replacing $t$ by $E_t$ would only insert some constant periods into the original process. As $\alpha$ gets closer to $1$, the constant periods vanish gradually.

\begin{figure}[h]
\begin{center}
\caption{Simulation of $u^*(t)$ for Example 1, $\alpha=.9$}\label{fig:EXAMPLE1_9}
\includegraphics[scale=.66]{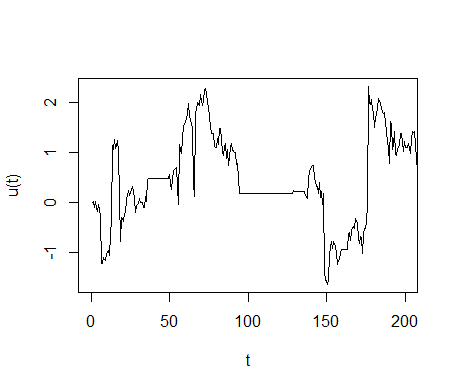}
\end{center}
\end{figure}

\begin{figure}[h]
\begin{center}
\caption{Simulation of $u^*(t)$ for Example 1, $\alpha=.7$}\label{fig:EXAMPLE1_7}
\includegraphics[scale=.66]{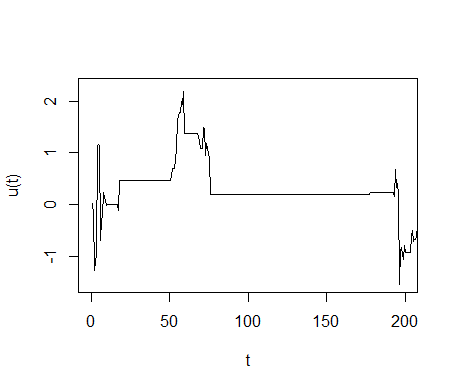}
\end{center}
\end{figure}

\begin{figure}[h]
\begin{center}
\caption{Simulation of $u^*(t)$ for Example 1, $\alpha=.5$}\label{fig:EXAMPLE1_5}
\includegraphics[scale=.66]{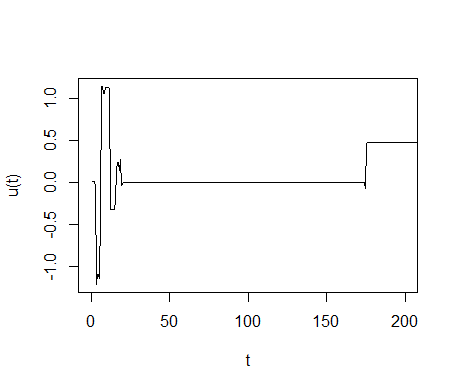}
\end{center}
\end{figure}

\end{exmp}

\section{A More General Time-changed Stochastic Control Problem }

Now we extend the time-changed SDE \eqref{simSDE} to a more general case by adding a time drift term as below, \begin{equation}
\begin{aligned}
dX(t)=&\mu(t, E_t, X(t-), u(t))dt+b(t, E_t, X(t-), u(t))dE_t+\sigma(t, E_t, X(t-), u(t))dB_{E_t}\\
&+\int_{|y|<c}\gamma(t, E_t, X(t-), u(t),y)\tilde{N}(dE_t,dy),
\end{aligned}
\end{equation}
with $X(0)=x_0\neq 0$,
where $\mu,b,\sigma,\gamma$ are real-valued functions satisfying the Lipschitz condition \ref{lip} and assumption \ref{tec}.

Suppose the performance function is given by
\begin{equation}\label{performance2}
J(u)=\mathbb{E} \Big[\int_0^Tf(t, E_t, X(t), u(t))dt+\int_0^Tg(t, E_t, X(t), u(t))dE_t+h(X(T))\Big], \ u\in \mathcal{A} ,
\end{equation}
where   the function $f, g:[0,T]\times \mathbb{R}_+ \times \mathbb{R} \times U   \rightarrow \mathbb{R}$ are continuous, $h:\mathbb{R}\rightarrow \mathbb{R}$ is $C^1, T<\infty $ is a fixed deterministic time and
\begin{equation}
\mathbb{E} \Big[\int_0^Tf(t, E_t, X(t), u(t))dt+\int_0^Tg(t, E_t, X(t), u(t))dE_t+h(X(T))\Big]<\infty, \forall u\in \mathcal{A}.
\end{equation}

The stochastic control problem is to find the optimal control $u^*\in \mathcal{A}$ such that
\begin{equation}\label{supjex}
J(u^*)=\sup_{u\in \mathcal{A}}J(u).
\end{equation}

\begin{rk}
Performance functions \eqref{performance1} and \eqref{performance2} are slightly different in terms of their integral kernels. This difference results in different $Hamiltonians$ and adjoint equations.
\end{rk}

Define  the  $Hamiltonian\ H\ :[0,T]\times\mathbb{R_+}\times\mathbb{R}\times U\times\mathbb{R}\times \mathbb{R}\times\mathcal{R}\to \mathbb{R}$ by

\begin{equation}
\begin{aligned}
H(t_1,t_2, x, u,p,q,r)=&\Big(p \mu(t_1,t_2, x, u)+f(t_1,t_2, x, u)\Big)\\
&+\Big(pb(t_1,t_2, x, u)+q\sigma(t_1,t_2, x, u)+g(t_1,t_2, x, u)\Big)\frac{dt_2}{dt_1}\\
&+\int_{\mathbb{R}}\gamma(t_1,t_2, x, u,z)r(t,z)\nu(dz)\frac{dt_2}{dt},
\end{aligned}
\end{equation}
or
\begin{equation}
\begin{aligned}
H(t,E_t, X(t),& u(t),p(t),q(t),r(t,z))=\Big(p(t)\mu(t,E_t, X(t), u(t))+f(t,E_t, X(t), u(t))\Big)\\
&+\Big(p(t)b(t,X(t),u(t))+q(t)\sigma(t,E_t, X(t), u(t))+g(t,E_t, X(t),u(t))\Big)\frac{dE_t}{dt}\\
&+\int_{\mathbb{R}}\gamma(t,E_t, X(t), u(t),z)r(t,z)\nu(dz)\frac{dE_t}{dt}.
\end{aligned}
\end{equation}

Define the adjoint equation
\begin{equation}\label{adj2}
\begin{aligned}
dp(t)=&-H_x(t,E_t,X(t),u(t),p(t),q(t),r(t,\cdot))dt\\
&+q(t)dB_{E_t}+\int_{\mathbb{R}}r(t,z)\tilde{N}(dE_t,dz), t<T\\
p(T)=&h_x(X(T))
\end{aligned}
\end{equation}

\begin{tm}\label{tcmptex}
(Time-Changed Maximum Principle Theorem) Let $\hat{u}\in \mathcal{A}$ with corresponding solution $\hat{X}=X^{(\hat{u})}$ and suppose there exists a solution $(\hat{p}(t), \hat{q}(t), \hat{r}(t,z))$ of the corresponding adjoint equation \eqref{adj1} satisfying
\begin{equation}
\mathbb{E}\Big[ \int_0^T (\hat{X}(t)-X^{(u)}(t))^2\Big(\hat{q}^2(t)+\int_{\mathbb{R}}\hat{r}^2(t,z)\nu(dz) \Big)dE_t\Big]<\infty
\end{equation}
and
\begin{equation}
\mathbb{E}\Big[ \int_0^T \hat{p}^2(t)\Big(\sigma^2(t,E_t,X^{(u)}(t),u(t) )+\int_{\mathbb{R}}\gamma^2(t,E_t,X^{(u)}(t),u(t),z)\nu(dz) \Big)dE_t\Big]<\infty, \forall u\in\mathcal{A}.
\end{equation}
Moreover, suppose that
\begin{equation}
H(t,E_t,\hat{X}(t),\hat{u}(t),\hat{p}(t), \hat{q}(t), \hat{r}(t,\cdot))=\sup_{v\in U} H(t,E_t,\hat{X}(t),v,\hat{p}(t), \hat{q}(t), \hat{r}(t,\cdot))
\end{equation}
for all $t>0$, that $h(x)$ in \eqref{performance2} is a concave function of $x$ and that
\begin{equation}
\hat{H}(x):=\max_{v\in U} H(t_1,t_2,x,v,\hat{p}(t), \hat{q}(t), \hat{r}(t,\cdot))
\end{equation}
exists and is a concave function of $x$ for all $t\in[0,T]$.
Then $\hat{u}$ is an optimal control of stochastic control problem \eqref{supjex}.
\end{tm}

\begin{proof}
Let $u \in \mathcal{A}$ be an admissible control with the  corresponding state process $X(t)=X^{(u)}(t)$. We would like to show that
\begin{equation}
\begin{aligned}
J(\hat{u})-J(u)=&\mathbb{E}\Big[ \int_0^T{f(t,E_t,\hat{X}(t),\hat{u}(t))-f(t,E_t,X(t),u(t))}dt\\
&+\int_0^T{g(t,E_t,\hat{X}(t),\hat{u}(t))-g(t,E_t,X(t),u(t))}dE_t+h(\hat{X}(T))-h(X(T)) \Big]\geq 0.
\end{aligned}
\end{equation}

Since $h$ is concave, using It\^o formula \eqref{itolevy},
\begin{equation}
\begin{aligned}
\mathbb{E}&[h(\hat{X}(T))-g(X(T))]\geq  \mathbb{E}[h_x(\hat{X}(T))(\hat{X}(T)-X(T))]=\mathbb{E}[(\hat{X}(T)-X(T))\hat{p}(T)]\\
=&\mathbb{E}\Big[ \int_0^T (\hat{X}(t)-X(t))d\hat{p}(t)+ \int_0^T \hat{p}(t) d(\hat{X}(t)-X(t))+ \int_0^T d\hat{p}(t) d(\hat{X}(t)-X(t))\Big]\\
=&\mathbb{E}\Big[ \int_0^T (\hat{X}(t)-X(t))d\hat{p}(t)+ \int_0^T \hat{p}(t) d(\hat{X}(t)-X(t))\\
&+ \int_0^T \hat{q}(t)\Big(\sigma(t,E_t,\hat{X}(t),\hat{u}(t))-\sigma(t,E_t,X(t),u(t))\Big)\hat{q}(t)dE_t\\
&+\int_0^T\int_{\mathbb{R}}\hat{r}(t,z)\Big(\gamma(t,E_t,\hat{X}(t),\hat{u}(t))-\gamma(t,E_t,X(t),u(t))\Big)\nu(dz)dE_t\Big].
\end{aligned}
\end{equation}

Among above terms,
\begin{equation}
\begin{aligned}
\mathbb{E}\Big[\int_0^T\hat{p}(t)d(\hat{X}(t)-X(t))\Big]=&\mathbb{E}\Bigg[\int_0^T \hat{p}(t)\Bigg( \Big(\mu(t,E_t,\hat{X}(t),\hat{u}(t))-\mu(t,E_t,X(t),u(t))\Big)dt\\
&+\Big(b(t,E_t,\hat{X}(t),\hat{u}(t))-b(t,E_t,X(t),u(t))\Big)dE_t\Bigg)\Bigg]
\end{aligned}
\end{equation}

Thus,
\begin{equation}
\begin{aligned}
J(\hat{u})-J(u)=&\mathbb{E}\Bigg[ \int_0^T (\hat{X}(t)-X(t))d\hat{p}(t)+ \int_0^T f(t,E_t,\hat{X}(t),\hat{u}(t))-f(t,E_t,X(t),u(t))dt\\
&+ \int_0^T g(t,E_t,\hat{X}(t),\hat{u}(t))-g(t,E_t,X(t),u(t))dE_t\\
&+\int_0^T \hat{p}(t)\Big[ \Big(\mu(t,E_t,\hat{X}(t),\hat{u}(t))-\mu(t,E_t,X(t),u(t))\Big)dt\\
&+\Big(b(t,E_t,\hat{X}(t),\hat{u}(t))-b(t,E_t,X(t),u(t))\Big)dE_t\Big]\\
&+ \int_0^T \hat{q}(t)\Big(\sigma(t,E_t,\hat{X}(t),\hat{u}(t))-\sigma(t,E_t,X(t),u(t))\Big)dE_t\\
&+\int_0^T\int_{\mathbb{R}}\hat{r}(t,z)\Big(\gamma(t,E_t,\hat{X}(t),\hat{u}(t))-\gamma(t,E_t,X(t),u(t))\Big)\nu(dz)dE_t\Bigg].
\end{aligned}
\end{equation}

In addition,
\begin{equation}
\begin{aligned}
&(H(t,E_t,\hat{X}(t),\hat{u}(t),\hat{p}(t),\hat{q}(t),\hat{r}(t))-H(t,E_t,X(t),u(t),\hat{p}(t),\hat{q}(t),\hat{r}(t)))dt\\
=&\Big[\hat{p}(t)\mu(t,E_t,\hat{X}(t),\hat{u}(t))-\hat{p}(t)\mu(t,E_t,X(t),u(t))+  f(t,E_t,\hat{X}(t),\hat{u}(t))-f(t,E_t,X(t),u(t))\Big]dt\\
&+  \Big(g(t,E_t,\hat{X}(t),\hat{u}(t))-g(t,E_t,X(t),u(t))\Big)dE_t+\Big(\hat{p}(t)b(t,E_t,\hat{X}(t),\hat{u}(t))+\hat{q}(t)\sigma(t,E_t,\hat{X}(t),\hat{u}(t))\Big)dE_t\\
&-\Big(\hat{p}(t)b(t,E_t,X(t),u(t))+\hat{q}(t)\sigma(t,E_t,X(t),u(t))\Big)dE_t\\
&+\int_{\mathbb{R}}\hat{r}(t,z)\Big( \gamma(t,E_t,\hat{X}(t),\hat{u}(t))- \gamma(t,E_t,X(t),u(t))\Big)\nu(dz)dE_t,
\end{aligned}
\end{equation}

and
\begin{equation}
\begin{aligned}
&(\hat{X}(t)-X(t))d\hat{p}(t)=\hat{X}(t)d\hat{p}(t)-X(t)d\hat{p}(t)\\
=&\hat{X}(t)\Big[-H_x(t,E_t,\hat{X}(t),\hat{u}(t),\hat{p}(t),\hat{q}(t),\hat{r}(t,))dt+\hat{q}(t)dB_{E_t}+\int_{\mathbb{R}}r(t,z)\tilde{N}(dE_t,dz)\Big]\\
&-X(t)\Big[-H_x(t,E_t,\hat{X}(t),\hat{u}(t),\hat{p}(t),\hat{q}(t),\hat{r}(t,))dt+\hat{q}(t)dB_{E_t}+\int_{\mathbb{R}}r(t,z)\tilde{N}(dE_t,dz)\Big]\\
=&-(\hat{X}(t)-X(t))H_x(t,E_t,\hat{X}(t),\hat{u}(t),\hat{p}(t),\hat{q}(t),\hat{r}(t,\cdot))dt\\
&+(\hat{X}(t)-X(t))\Big(\hat{q}(t)dB_{E_t}+ \int_{\mathbb{R}}\hat{r}(t,z)\tilde{N}(dE_t,dz)\Big).\\
\end{aligned}
\end{equation}

Then, by concavity of $H$ and following the proof in \cite{fos},
\begin{equation}
\begin{aligned}
J(\hat{u})-J(u)=&\mathbb{E}\Big[ \int_0^T -(\hat{X}(t)-X(t))H_x(t,E_t,\hat{X}(t),\hat{u}(t),\hat{p}(t),\hat{q}(t),\hat{r}(t,\cdot))dt\\
&+\int_0^TH(t,E_t,\hat{X}(t),\hat{u}(t),\hat{p}(t),\hat{q}(t),\hat{r}(t,\cdot))-H(t,E_t,X(t),u(t),\hat{p}(t),\hat{q}(t),\hat{r}(t,\cdot))dt\Big]\\
&\geq 0.
\end{aligned}
\end{equation}
\end{proof}

\begin{exmp}(Income and Consumption Optimization)
Consider the stochastic control problem
\begin{equation}
\Phi(x_0)=\sup_{u\in \mathcal{A}} \mathbb{E}\Big[\int_0^{\tau}\exp(-\delta t) u(t)^2dt \Big],
\end{equation}
where
\begin{equation}\tau=\inf\{t>0; X(t)\leq 0\}\end{equation}
and
\begin{equation}
dX(t)=-u(t) dt+X(t)\Big(bdE_t+\sigma dB_{E_t}+\theta\int_\mathbb{R}z\tilde{N}(dz,dE_t)\Big),\ X(0)=x_0>0,
\end{equation}
where $\delta>0, \sigma,$ and $\theta$ are constants and $b=-\frac{\sigma^2+\theta^2\int_\mathbb{R}z^2\nu(dz)}{2}$.

We can interpret $u(t)$ as the consumption rate, $X(t)$ as the corresponding wealth, and $\tau$ as the bankruptcy time. Then $\Phi$ represents the maximal expected total quadratic utility of the consumption up to bankruptcy time.

Define the $Hamiltonian\ H:$
\begin{equation}
H(t)=-p(t)u(t)+\exp(-\delta t)u(t)^2+X(t)\Big(p(t)b+q(t)\sigma+\int_\mathbb{R}\theta z r(t,z)\nu(dz)\Big)\frac{dE_t}{dt},
\end{equation}
and the adjoint equation
\begin{equation}\label{adje1}
\begin{aligned}
dp(t)=&-\Big(p(t)b+q(t)\sigma+\int_\mathbb{R}\theta z r(t,z)\nu(dz)\Big)dE_t\\
&+q(t)dB_{E_t}+\int_{\mathbb{R}}r(t,z)\tilde{N}(dE_t,dz), t<\tau,\\
p(T)=&0.
\end{aligned}
\end{equation}

Let $\frac{\partial H}{\partial u}=(-p(t)+2u(t)\exp(-\delta t))=0$, we have $u^*(t)=\frac{p(t)}{2}\exp(\delta t)$. Suppose that $p(t)=h(t)X(t)$, then $u^*(t)=\frac{h(t)X(t)}{2}\exp(\delta t)$, thus
\begin{equation}\label{adje2}
\begin{aligned}
dp(t)&=X(t)h(t)'dt+h(t)dX(t)\\
&=X(t)h(t)'dt+(-u(t)h(t))dt+h(t)X(t)\Big(bdE_t+\sigma dB_{E_t}+\theta\int_\mathbb{R}z\tilde{N}(dz,dE_t)\Big)\\
&=X(t)\Big(h(t)'-\frac{h(t)}{2}\exp(\delta t) \Big)dt+h(t)X(t)\Big(bdE_t+\sigma dB_{E_t}+\theta\int_\mathbb{R}z\tilde{N}(dz,dE_t)\Big)
\end{aligned}
\end{equation}

Comparing \eqref{adje1} and \eqref{adje2}, we derive that $h'(t)=\frac{h(t)}{2}e^{\delta t}$, equivalently, $h(t)=\exp(\frac{1}{2\delta}e^{\delta t}),$ thus
\begin{equation}
u(t)^*=\exp(\frac{1}{2\delta}e^{\delta t}+\delta t)\frac{X(t)}{2}.
\end{equation}
Moreover,
\begin{equation}
\begin{aligned}
h(t)X(t)\sigma&=q(t),\\
h(t)X(t)\theta z&=r(t,z).
\end{aligned}
\end{equation}
Some algebra implies that
\begin{equation}
\begin{aligned}
q(t)&=2\exp(-\delta t)u(t)\sigma,\\
r(t,z)&=2\exp(-\delta t)u(t)\theta z.\\
\end{aligned}
\end{equation}

 A simulation of the optimal control $u^*(t)$ with $\delta=-.001,  \sigma = 1, \theta=1,  x_0=1$, standard normal distribution $\nu$, and inverse stable subordinator $E(t)$ having $\alpha = .9$ is displayed in Figure \ref{fig:figure2home}.

\begin{figure}
\begin{center}
\caption{Simulation of $u^*(t)$ for Example 2}
\includegraphics[scale=.66]{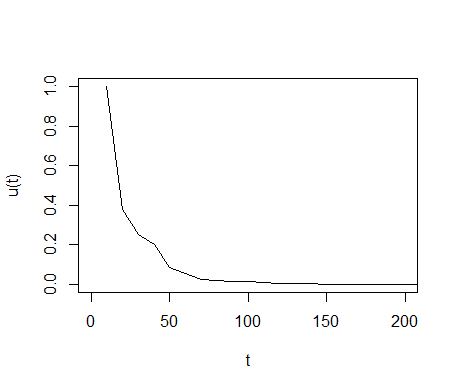}
\label{fig:figure2home}
\end{center}
\end{figure}

Because of the existence of $dt$ term in the underlying process $X(t)$, the simulated process $u^*(t)$ has no periods of constant value. Compared with $dE_t$ terms, $dt$ term plays the dominating role in the evolution of corresponding wealth $X(t)$, see \cite{erni3} for a detailed discussion. More specifically, the increasing trend $bX(t)dE_t$ is dominated by the consumption rate $-u(t)dt$. Consequently, the optional consumption rate declines as the wealth shrinks in the long term.

\end{exmp}


\newpage

\begin{thebibliography}{99}

\bibitem{bens} A. Bensoussan, Maximum Principle and Dynamic Programming Approaches
to the Optimal Control of Partially Observed Diffusions. Stochastics: An International Journal of Probability and Stochastic Processes 9.3 (1983), 169-222.

\bibitem{bism} J. M. Bismut, Conjugate convex functions in optimal stochastic control. J. Math. Anal. Appl. 44 (1973), 384-404.

\bibitem{econ} W. H. Fleming, T. Pang,  An application of stochastic control theory to financial economics. SIAM J. Control, 43(2) (2004), pp.502-531.

\bibitem{fos} N. C. Framstad, B. {\O}ksendal, A. Sulem, Sufficient stochastic maximum principle for the optimal control of jump diffusions and applications to finance. J. Optim. Theory Appl. 124 (2005), no. 2, 511-512.

\bibitem{jawy}  J. Janczura and A. Wylomanska. Subdynamics of financial data from fractional Fokker-Planck equation, Acta Physica Polonica B 40 (2009) 1341-1351.

\bibitem{keib} M. Hahn, K. Kobayashi, S. Umarov, SDEs driven by a time-changed L\'evy process and their associated time-fractional order pseudo-differential equations.J. Theoret. Probab. 25 (2012), no. 1, 262-279.

\bibitem{keib1} E. Jum, K. Kobayashi, A strong and weak approximation scheme for stochastic differential equations driven by a time-changed Brownian motion. Probab. Math. Statist. 36 (2016), no. 2, 201--220.

\bibitem{kush} H. J.  Kushner, Necessary conditions for continuous parameter stochastic optimization problems. SIAM J. Control 10 (1972), 550-565.

\bibitem{marc2} M. Magdziarz, T. Zorawik, Stochastic representation of a fractional subdiffusion equation. The case of infinitely divisible waiting times, L\'evy noise and space-time-dependent coefficients. Proc. Amer. Math. Soc. 144 (2016), no. 4, 1767-1778.

\bibitem{meer} M. M. Meerschaert, P. Straka, Inverse stable subordinators. Math. Model. Nat. Phenom. 8 (2013), no. 2, 1-16.

\bibitem{erni3} E. Nane, Y. Ni, Path stability of stochastic differential equations driven by time-changed L\`evy noises. ALEA Lat. Am. J. Probab. Math. Stat. 15 (2018), no. 1, 479--507.

\bibitem{erni2} E. Nane, Y. Ni, Stability of the solution of stochastic differential equation driven by time-changed L\'evy noise. Proc. Amer. Math. Soc. 145 (2017), no. 7, 3085-3104.

\bibitem{erni} E. Nane, Y. Ni, Stochastic solution of fractional Fokker-Planck equations with space-time-dependent coefficients. J. Math. Anal. Appl. 442 (2016), no. 1, 103-116.

\bibitem{peng} S. G. Peng, A general stochastic maximum principle for optimal control problems. SIAM J. Control Optim. 28 (1990), no. 4, 966-979.

\bibitem{fin}   W. J. Runggaldier, On stochastic control in finance. In Mathematical Systems Theory in Biology, Communications, Computation, and Finance (pp. 317-344) (2003). Springer, New York, NY.

\bibitem{bio} D. Wilkinson, Stochastic Modelling for Systems Biology. New York: Chapman and Hall/CRC. (2006).

\bibitem{wu} Q. Wu, Stability analysis for a class of nonlinear time-changed systems. Cogent Math. 3 (2016), Art. ID 1228273, 10 pp.

\end{thebibliography}
\end{document}